\documentclass[10pt]{amsart}						

	\usepackage{graphicx, subfig, color}	
	
	\usepackage[pagebackref,  pdfpagemode=FullScreen,  colorlinks=true]{hyperref}	   

	\newtheorem{thm}{Theorem}[section]
  	\newtheorem{cor}{Corollary}[section]
  	
  	\newtheorem{prop}{Proposition}[section]

	\theoremstyle{definition}

	\newcommand{\M}{\mathcal{M}}
	\newcommand{\Mbar}{\overline{\mathcal{M}}}

	\makeatletter
	\@namedef{subjclassname@2010}{%
	\textup{2010} Mathematics Subject Classification}
	\makeatother

\begin{document}

\title{Double total ramifications for curves of genus $2$}
\author{Nicola Tarasca}
\email{tarasca@math.utah.edu}
\address{University of Utah, Department of Mathematics, 155 S 1400 E, Salt Lake City, UT 84112, USA}

\begin{abstract}
Inside the moduli space of curves of genus $2$ with $2$ marked points, we consider the loci of curves admitting a map to $\mathbb{P}^1$ of degree $d$ totally ramified over the two marked points, for $d\geq 2$. Such loci have codimension two. We compute the class of their closures in the moduli space of stable curves. Several results will be deduced from this computation. 
\end{abstract}

\maketitle

A classical way of producing subvarieties of moduli spaces of curves is by means of Hurwitz theory. Let $\M_g$ be the moduli space of smooth curves of genus $g$. If $1<d<(g+2)/2$, the locus of curves of genus $g$ admitting a map to $\mathbb{P}^1$ of degree $d$ has codimension $g-2d+2$ inside $\M_g$ (\cite{MR626954}).

From such loci, one can produce subvarieties of the moduli spaces $\M_{g,n}$ of smooth $n$-pointed curves of genus $g$ by imposing exceptional ramifications at the marked points. An important example is the {\it double ramification locus}, that is, the locus of curves admitting a map to $\mathbb{P}^1$ with prescribed ramification profile over two points and simple ramifications elsewhere. 
Such a locus has codimension $g$ in $\M_{g,n}$.
 
Another description is the following. Fix a multi-index $\underline{d}=(d_1,\dots,d_n)\in \mathbb{Z}^n$ of degree $0$, that is, $\sum_i d_i=0$.
Let $\mathcal{J}_{g,n}\rightarrow \M_{g,n}$ be the universal Jacobian variety over $\M_{g,n}$ and consider the section $\varphi_{\underline{d}}\colon \M_{g,n}\rightarrow \mathcal{J}_{g,n}$ defined by $\varphi_{\underline{d}}([C,p_1,\dots,p_n])=\mathcal{O}_C(\sum_{i} d_i p_i)$. 
The double ramification locus $\mathcal{DR}_g(\underline{d})$ in $\M_{g,n}$ is 
the pull-back of the zero section $\mathcal{Z}_g$ in $\mathcal{J}_{g,n}$ via $\varphi_{\underline{d}}$.

Double ramification loci have played a crucial role in the study of topological recursive relations in \cite{MR1908062}, and in the proof of the $r$-spin Witten conjecture in \cite{MR2722511}. 
The problem of computing the classes of closures of double ramification loci in (possibly partial) compactifications of $\M_{g,n}$ is generally known as the Eliashberg's problem, and its interest is heightened by applications in symplectic field theory. 

The universal Jacobian variety $\mathcal{J}_{g,n}$ over $\M_{g,n}$ can be extended over the moduli space $\M_{g,n}^{ct}$ of stable curves of compact type. Given a multi-index $\underline{d}\in\mathbb{Z}^n$ of degree $0$, Hain computed the class in $H^{2g}(\M_{g,n}^{ct},\mathbb{Q})$ of the pull-back of the zero section ${\mathcal{Z}}^{ct}_g$ via the map $\varphi_{\underline{d}}\colon \M_{g,n}^{ct}\rightarrow \mathcal{J}_{g,n}^{ct}$ using normal functions (\cite{Hain}).

Moreover, Hain's formula holds in the Chow group $A^g(\M_{g,n}^{ct})$ (see for instance \cite{GZ1}), and Grushevsky and Zakharov have extended it over the open subset of $\Mbar_{g,n}$ parametrizing curves with at most one non-disconnecting node (\cite{MR3189435}).

Note however that the pull-back of the zero section of the universal Jacobian variety $\mathcal{J}_{g,n}^{ct}\rightarrow \M_{g,n}^{ct}$ is not irreducible and contains other components besides the closure of the double ramification locus. 

This phenomenon has been investigated in genus $1$ in \cite{MR3231020}. For $\underline{d}\in\mathbb{Z}^n$ such that all $d_i$ are non-zero, one has
\[
\left[ \varphi_{\underline{d}}^* \overline{\mathcal{Z}}_1\right] = \left[\overline{\mathcal{DR}}_1(\underline{d})\right] +\delta_{0,\{1,\dots,n\}}\,\,\in\textrm{Pic}_\mathbb{Q}(\Mbar_{1,n}),
\]
where $\delta_{0,\{1,\dots,n\}}$ is the class of the divisor of curves with a rational tail containing all marked points.

The study of the closure of Hurwitz loci in the moduli space of stable pointed curves $\Mbar_{g,n}$ can be carried out using the theory of admissible covers (\cite{MR664324}). For instance, in $\Mbar_g$ explicit expressions are known for the classes of the closures of Hurwitz divisors (\cite{MR664324}), and Hurwitz loci of codimension two (\cite{MR3109733}).

In this paper, we use the theory of admissible covers to study the closure of double ramification loci, considered for $g=1$ in \cite{MR3231020}, in the next non-trivial case, that is, the case $g=n=2$.

Let ${\mathcal{DR}}_2(d)$ be the locus of pointed curves $[C,p_1,p_2]$ in $\M_{2,2}$ admitting a map to $\mathbb{P}^1$ of degree $d$ totally ramified at $p_1$ and at $p_2$. In other words,
\[
{\mathcal{DR}}_2(d) = \{[C,p_1,p_2]\in\M_{2,2} :  dp_1\equiv dp_2 \in\textrm{Pic}^d(C)\}.
\]
We obtain an explicit formula for the class of its closure in $\Mbar_{2,2}$.

\begin{thm}
\label{ACC}
For $d\geq 2$, the class of $\overline{\mathcal{DR}}_2(d)$ in $A^2(\Mbar_{2,2})$ is
\begin{multline*}
\left[\overline{\mathcal{DR}}_2(d)\right] = (d^2-1)\Bigg[
\frac{d^2}{2}\psi_1\psi_2 +\frac{2-d^2}{4} \left(\psi_1^2+\psi_2^2\right)\\
+\left(\psi_1+\psi_2\right)\left(-\frac{3d^2+2}{20} \delta_{1,1}
+\frac{d^2-6}{10}\left(\delta_{1,2}+\frac{1}{12}\delta_0 \right)\right)
\Bigg].
\end{multline*}
\end{thm}
\noindent Throughout, we use the same notation for divisor classes and codimension-two classes of $\Mbar_{2,2}$ as in \cite{MR1672112}. Note that the Picard group of $\Mbar_{2,2}$ is generated by the following classes: $\psi_1$, $\psi_2$ are the cotangent line classes at the two marked points; $\delta_0$ is the class of the closure of the divisor $\Delta_0$ of nodal irreducible curves; $\delta_2$ is the class of the closure of the divisor $\Delta_2$, whose general element has a component of genus $2$ meeting transversally a rational component with the two marked points; $\delta_{1,1}$ is the class of the closure of the divisor $\Delta_{1,1}$ consisting of curves having two elliptic components meeting transversally, each with a marked point; finally $\delta_{1,2}$ is the class of the closures of the divisor $\Delta_{1,2}$ of  curves consisting of two elliptic components meeting transversally, with the two markings on just one component.

By \cite{MR2120989}, the class of $\overline{\mathcal{DR}}_2(d)$ lies in the tautological group $R^2(\Mbar_{2,2})\subset A^2(\Mbar_{2,2})$. By \cite{MR1672112}, the group $R^2(\Mbar_{2,2})$ is generated by product of divisor classes. In order to determine the coefficients of the class of $\overline{\mathcal{DR}}_2(d)$, we combine relations coming from test surfaces (\S \ref{testsurfaces}), symmetries of $\overline{\mathcal{DR}}_2(d)$ (\S \ref{symmetry}), and the push-forward to $\Mbar_{2,1}$ (\S\S \ref{div} and \ref{push}).

As a first consequence of the explicit expression of the class of $\overline{\mathcal{DR}}_2(d)$, we have the following result, shown in \S \ref{codim2cone}.

\begin{cor}
\label{CI}
The locus $\overline{\mathcal{DR}}_2(d)$ is not a complete intersection in $\Mbar_{2,2}$.
\end{cor}

This last result shows that the class of $\overline{\mathcal{DR}}_2(d)$ cannot be expressed as a product of two effective divisor classes. On the other hand, we recall that 
\[
 \left[\varphi^*_{(d,-d)}{\mathcal{Z}}^{ct}_2\right] = \frac{1}{2}\left[\varphi_{(d,-d)}^*\mathcal{T} \right]^2 \in A^2(\M^{ct}_{2,2}),
\]
where $\mathcal{T}$ is the universal symmetric theta divisor trivialized along the zero section (see for instance \cite{GZ1}). 

As we have remarked, the restriction of $\overline{\mathcal{DR}}_2(d)$ to the locus of curves of compact type is one of the components of the pull-back of the zero section ${\mathcal{Z}}^{ct}_2$ via $\varphi_{(d,-d)}$.
Using the computation of the class $\varphi_{(d,-d)}^*\mathcal{T}$ from \cite{Hain} and \cite{GZ1}, we deduce in \S \ref{Hain} the following.

\begin{cor}
\label{Hain-AC}
We have the following equality in $A^2(\mathcal{M}^{ct}_{2,2})$
\begin{eqnarray}
\label{HAC}
\left[\varphi^*_{(d,-d)}{\mathcal{Z}}^{ct}_2\right] = \left[\overline{\mathcal{DR}}_2(d) \right]+ \delta_{22}+(2d^2-1)\delta_{11|}+ \left(d^2-\frac{6}{5}\right)\delta_{11|12}.
\end{eqnarray}
\end{cor}

Here the class $\delta_{22}$ is the push-forward of the $\psi$-class in $\Mbar_{2,1}$ via the glueing map $\Mbar_{2,1}\times \Mbar_{0,3}\rightarrow \Delta_2\subset\Mbar_{2,2}$;
the class $\delta_{11|}$ is the class of the closure of the locus of curves with two elliptic tails attached at a rational component containing the two marked points;
finally, we have $\delta_{11|12}=\delta_{1,2}\cdot\delta_2$.

The set $\{\delta_{22}$, $\delta_{11|}$, $\delta_{11|12}\}$ extends to a basis of the $5$-dimensional space $R^2(\mathcal{M}^{ct}_{2,2})$ by means of the classes $\delta_{11|1}$, $\delta_{11|2}$ of the two components of the intersection $\delta_{1,1}\cdot\delta_{1,2}$: for $i\in\{1,2\}$, $\delta_{11|i}$ is the class of the closure of the locus of curves with a central rational component and two elliptic tails, with the point $i$ on one of the elliptic components, and the other point on the central rational component.

We remark that, as in the $g=1$ case, the difference of the pull-back of the zero section and the closure of the double ramification locus in $\M^{ct}_{2,2}$ is supported on loci of curves with a rational component containing the two marked points. 

Note that the class appearing on the left-hand side of (\ref{HAC}) has one more interpretation: by the results of \cite{MR2864866} and \cite{Marcus-Wise}, the class of the pull-back of the zero section coincides with the push-forward in $A^{g}(\M_{g,n}^{ct})$ of the virtual fundamental class of the space of relative stable maps to an unparameterized rational curve. See also \cite{RPnotes2013} for a discussion on different geometric ways to extend such classes to $\Mbar_{g,n}$. 

By varying $d$, Theorem \ref{ACC} gives us infinitely many rays in the pseudo-effective cone of codimension-two classes in $\Mbar_{2,2}$.
A natural problem is to study the cone spanned by the classes of $\overline{\mathcal{DR}}_2(d)$ for $d\geq 2$. Let $\left[\overline{\mathcal{DR}}_2(\infty)\right]$ be the class defined as follows:
\begin{eqnarray*}
\left[\overline{\mathcal{DR}}_2(\infty)\right] \!\!&:=&\!\! \lim_{d\rightarrow \infty}\quad \frac{\left[\overline{\mathcal{DR}}_2(d)\right]}{d^4}\\
&=&\!\! \frac{1}{2}\psi_1\psi_2 -\frac{1}{4} \left(\psi_1^2+\psi_2^2\right)
+\left(\psi_1+\psi_2\right)\left(-\frac{3}{20} \delta_{1,1}
+\frac{1}{10}\delta_{1,2}+\frac{1}{120}\delta_0 \right).
\end{eqnarray*}

\begin{cor}
\label{cone}
The classes of the loci $\overline{\mathcal{DR}}_2(d)$ for $d\geq 2$ lie in the two-dimensional cone spanned by the classes $\left[\overline{\mathcal{DR}}_2(2)\right]$ and $\left[\overline{\mathcal{DR}}_2(\infty)\right]$.
\end{cor}

In \cite{MR3231020} Chen and Coskun show that infinitely many double ramification classes are extremal and rigid in the pseudo-effective cone of divisors of $\Mbar_{1,n}$ for $n\geq 3$, hence proving that $\Mbar_{1,n}$ is not a Mori dream space for $n\geq 3$. In \S\ref{codim2cone}, we show that for $d\geq 3$ the class of $\overline{\mathcal{DR}}_2(d)$ is {\it not extremal} in the pseudo-effective cone of codimension-two classes of $\Mbar_{2,2}$. In \S \ref{div}, we obtain stronger results for the push-forward of the classes of $\overline{\mathcal{DR}}_2(d)$ via the map $\pi_i\colon \Mbar_{2,2}\rightarrow\Mbar_{2,1}$ forgetting the point $i$: the divisor class $(\pi_i)_*\left[\overline{\mathcal{DR}}_2(d)\right]$ is {\it extremal and rigid} for $d=2$, and {\it big and moving} for $d\geq 3$.

In this paper all classes are stack fundamental classes, and all cohomology and Chow groups are taken with rational coefficients.

{\bf Acknowledgments.}
I would like to thank Renzo Cavalieri, Steffen Marcus, and Nicola Pagani for helpful discussions on double ramification loci, and the referee for suggesting several improvements in the exposition.


\section{Forgetting a point}
\label{div}

Let $\pi_i\colon \Mbar_{2,2}\rightarrow \Mbar_{2,1}$ be the map forgetting the point $i$, for $i=1,2$. In this section, we study the closures in $\Mbar_{2,1}$ of the divisors
\[
 (\pi_i)_*\left({\mathcal{DR}}_2(d)\right)=\{[C,p]\in\M_{2,1} : \exists \, x\in C\setminus\{p\} \,\, \mbox{such that}\,\, dp\equiv dx\}
\]
for $d\geq 2$. Note that we have
\[
 (\pi_1)_*\left(\overline{\mathcal{DR}}_2(d)\right)=(\pi_2)_*\left(\overline{\mathcal{DR}}_2(d)\right),
\]
since the loci $\overline{\mathcal{DR}}_2(d)$ are symmetric in the two marked points. The classes of these divisors will be one of the ingredients in the proof of Theorem \ref{ACC} in the next section.

We will need the following result about the enumerative geometry of pencils on the general curve, in the spirit of \cite[Theorem A,B]{MR664324}, \cite[\S 2]{MR735335}, \cite[\S 2.1]{MR2574363}.

\begin{prop}
\label{m}
 Let $(C,p)$ be a general pointed curve of genus $g\geq 1$. The number of pairs $(L,x)\in W^1_{g+1}(C)\times C$ satisfying the conditions 
\[
 h^0(L\otimes \mathcal{O}(-(g+1)x))\geq 1 \quad \textrm{and} \quad h^0(L\otimes \mathcal{O}(-2p))\geq 1 
\]
is
\[
 m(g):=(g+2)g^2.
\]
\end{prop}

\begin{proof}
By \cite[Theorem B]{MR664324}, a smooth Brill-Noether general curve $\widetilde{C}$ of genus $g+1$ has 
\[
 g(g+1)(g+2)
\]
pairs $(L,x)\in W^1_{g+1}(\widetilde{C})\times \widetilde{C}$ such that $h^0(L\otimes \mathcal{O}(-(g+1)x))\geq 1$.

 Let $(C,p)$ be a general pointed curve of genus $g\geq 1$, and let us consider the curve obtained from $C$ by attaching an elliptic tail $(E,p)$ at the point $p\in C$. Since $p$ is general in $C$, the curve $C\cup_p E$ is a Brill-Noether general curve.  It follows that the curve $C\cup_p E$ admits $g(g+1)(g+2)$ admissible covers of degree $g+1$ totally ramified at a certain point $x\in C\cup_p E$. 

We distinguish two cases. If $x\in E$, then the admissible cover is totally ramified at $x$ and $p$, the restriction to $C$ is uniquely determined by $|\mathcal{O}((g+1)p)|$, and $p-x$ is a non-trivial $(g+1)$-torsion point in $\textrm{Pic}^0(E)$. There are $(g+1)^2-1$ choices for $x$, and each determines a unique admissible cover. If $x\in C$, then
the restriction of the admissible cover to $C$ is totally ramified at $x$ and has a simple ramification at $p$. That is, it corresponds to a pair $(L,x)$ as in the statement. The restriction to $E$ is uniquely determined by $|\mathcal{O}(2p)|$. It follows that
\[
  g(g+1)(g+2) = ((g+1)^2-1)+m(g),
\]
hence the statement
\end{proof}

We are now ready to study the classes of the divisors $(\pi_i)_*\left(\overline{\mathcal{DR}}_2(d)\right)$. We recall that the classes $\psi,\delta_0,\delta_1$ form a basis for $\textrm{Pic}(\Mbar_{2,1})$, and we have the following equality $\lambda=\frac{1}{10}\delta_0+\frac{1}{5}\delta_1$.
Let $\mathcal{W}$ be the Weierstass divisor in $\M_{2,1}$
\[
 \mathcal{W} = \{[C,p]\in\M_{2,1}: p\in C \,\,\mbox{is a Weierstrass point}\}.
\]
The class of the closure of $\mathcal{W}$ in $\Mbar_{2,1}$ is
\[
\left[\overline{\mathcal{W}}\right]=3\psi-\frac{1}{10}\delta_0-\frac{6}{5}\delta_1
\]
(see \cite[Theorem 2.2]{MR910206}). If $d=2$, it is easy to see that the class of the locus $(\pi_i)_*\left( \overline{\mathcal{DR}}_2(2) \right)$ is $5\cdot \left[\overline{\mathcal{W}} \right]$. Indeed, for every smooth pointed curve $(C,p)$ of genus $2$ with $p\in C$ a Weierstrass point, there are $5$ points in ${\mathcal{DR}}_2(2)$ lying over $(C,p)$, and they correspond to the other $5$ Weierstrass points of $C$. Using admissible covers, stable pointed curves also have $5$ points lying over them.

When $d=3$, the class of the locus $(\pi_i)_*\left( \overline{\mathcal{DR}}_2(3) \right)$ has been computed by Farkas in \cite[Proposition 4.1]{MR2574363}. The computation in the case $d\geq 3$ is a straightforward generalization of the case $d=3$. In the following, we work out the details.

In \cite{MR791679} Diaz studied the closure in $\Mbar_g$ of the following divisor of exceptional Weierstrass points
\[
 \mathfrak{Di}_{g-1}=\{[C]\in\M_g: \exists \, x\in C \,\,\mbox{such that}\,\,h^0(\mathcal{O}_C(g-1)x)\geq 2 \}
\]
for $g\geq 3$, and computed its class in $\textrm{Pic}(\Mbar_g)$:
\[
 \overline{\mathfrak{Di}}_{g-1}\equiv \frac{g^2(g-1)(3g-1)}{2}\lambda-\frac{(g-1)^2 g (g+1)}{6}\delta_0-\sum_{i\geq 1}\frac{i(g-i)g(g^2+g-4)}{2}\delta_i.
\]
For $d\geq 2$, let $\chi_d\colon \Mbar_{2,1}\rightarrow \Mbar_{d+1}$ be the map obtained by attaching a fixed general pointed curve of genus $d-1$ at the marked point. 
 
\begin{prop} 
\label{Di}
For $d\geq 2$, we have
\[
 \chi_d^*\left(\overline{\mathfrak{Di}}_d\right) \equiv (\pi_i)_*\left( \overline{\mathcal{DR}}_2(d) \right) + (d+1)(d-1)^2 \cdot\overline{\mathcal{W}}\,\,\in \mbox{\rm Pic}(\Mbar_{2,1}).
\]
\end{prop}

\begin{proof}
 The proof is identical to the $d=3$ case shown in \cite{MR2574363}. Let $[C_2,p]$ be a point in $\Mbar_{2,1}$, and let $[C_{d-1},p]$ be a general pointed curve of genus $d-1$. Suppose that $[C_2\cup_p C_{d-1}]$ is contained in $\chi_d^*\left(\overline{\mathfrak{Di}}_d\right)$. Then, there exists an admissible cover of $C_2\cup_p C_{d-1}$ of degree $d$ totally ramified at a smooth point $x\in C_2\cup_p C_{d-1}$. There are two cases. If $x\in C_{d-1}$, then the admissible cover has a simple ramification at the point $p$, the restriction on $C_2$ is uniquely determined by $|\mathcal{O}(2p)|$, $p\in C_2$ is a Weierstrass point, and by Proposition \ref{m} there are $m(d-1)$ choices for the restriction of the cover on $C_{d-1}$. If $x\in C_2$, then the admissible cover is totally ramified at $x$ and $p$, hence the restriction on $C_2$ is uniquely determined, and the restriction on $C_{d-1}$ is uniquely determined by $|\mathcal{O}(dp)|$. 
\end{proof}

\begin{cor}
\label{pfwdclass}
For $d\geq 2$ we have
\[
(\pi_i)_*\left[ \overline{\mathcal{DR}}_2(d) \right] = (d^2-1)\left((d^2+1)\psi-\frac{d^2+6}{5}\left(\frac{1}{12}\delta_0+\delta_1\right) \right).
\]
\end{cor}

\begin{proof}
 The statement follows from Proposition \ref{Di} using the following well known formulae: $\chi^*(\lambda)=\lambda$, $\chi^*(\delta_0)=\delta_0$, $\chi^*(\delta_1)=\delta_1$, and $\chi^*(\delta_2)=-\psi$.
\end{proof}

In \cite{MR2216265} Rulla shows that the pseudo-effective cone of divisor classes in $\Mbar_{2,1}$ is generated by the classes $\left[\overline{\mathcal{W}}\right]$, $\delta_0$, $\delta_1$; the nef cone is generated by the classes $\psi$, $\lambda=\frac{1}{10}\delta_0+\frac{1}{5}\delta_1$, and $12\lambda-\delta_0$; finally, the moving cone is generated by the nef cone together with the classes $D=30\left(\left[\overline{\mathcal{W}}\right]+\psi\right)$ and $E=20\left[\overline{\mathcal{W}}\right]+3\delta_0+6\delta_1$. 

\begin{cor}
The classes $(\pi_i)_*\left[ \overline{\mathcal{DR}}_2(d) \right]$ lie in the two-dimensional cone spanned by the classes $(\pi_i)_*\left[ \overline{\mathcal{DR}}_2(2) \right]=5\left[\overline{\mathcal{W}}\right]$ and $(\pi_i)_*\left[ \overline{\mathcal{DR}}_2(\infty) \right]=\frac{1}{6}\left[\overline{\mathcal{W}}\right]+\frac{1}{2}\psi$ inside the pseudo-effective cone of divisor classes in $\Mbar_{2,1}$. In particular, the class $(\pi_i)_*\left[ \overline{\mathcal{DR}}_2(d) \right]$ is extremal and rigid for $d=2$, and big and moving for $d>2$.
\end{cor}

\begin{proof}
It is easy to verify the following two equalities
\begin{eqnarray*}
 (\pi_i)_*\left[ \overline{\mathcal{DR}}_2(d) \right] &=& (d^2-1) \left(\frac{d^2+6}{6} \left[\overline{\mathcal{W}}\right] + \frac{d^2-4}{2}\psi\right) \\
&=& (d^2-1)\left(\frac{1}{3}(\pi_i)_*\left[ \overline{\mathcal{DR}}_2(2) \right] +(d^2-4) (\pi_i)_*\left[ \overline{\mathcal{DR}}_2(\infty) \right]\right).
\end{eqnarray*}
Note that this last equality is the push-forward of equality (\ref{2dimcone}).
It follows that the classes $(\pi_i)_*\left[ \overline{\mathcal{DR}}_2(d) \right]$ are in the cone spanned by the classes $(\pi_i)_*\left[ \overline{\mathcal{DR}}_2(2) \right]$ and $(\pi_i)_*\left[ \overline{\mathcal{DR}}_2(\infty) \right]$. Since the class $\psi$ is big, from the first equality above we have that the classes $(\pi_i)_*\left[ \overline{\mathcal{DR}}_2(d) \right]$ are big for $d\geq 3$. Note that, for $d=3$, the class $(\pi_i)_*\left[ \overline{\mathcal{DR}}_2(3) \right]$ is a multiple of the class $D$, and for $d\geq 3$ the classes $(\pi_i)_*\left[ \overline{\mathcal{DR}}_2(d) \right]$ lie in the cone spanned by $D$ and $\psi$, hence by Rulla's result, they are in the boundary of the moving cone.
\end{proof}

\section{Counting admissible covers}
\label{adm}

By \cite{MR2120989}, the classes of the loci $\overline{\mathcal{DR}}_g(\underline{d})$ lie in the tautological group $R^g(\Mbar_{g,n})\subset A^g(\Mbar_{g,n})$. By the results of \cite{MR1672112}, we have that the rational cohomology ring of $\Mbar_{2,2}$ is generated by $\textrm{Pic}_{\mathbb{Q}}(\Mbar_{2,2})$, hence in particular $RH^2(\Mbar_{2,2})=H^4(\Mbar_{2,2},\mathbb{Q})$.  

The Picard group $\textrm{Pic}(\Mbar_{2,2})$ is generated by the classes $\psi_1, \psi_2, \delta_0, \delta_2, \delta_{1,1}, \delta_{1,2}$, defined in the introduction.
Getzler shows that the relations in $RH^2(\Mbar_{2,2},\mathbb{Q})$ amongst products of divisor classes are spanned by
\begin{eqnarray}
\label{relations}
 \delta_{1,2}\left(12\delta_{1,1}+12\delta_{1,2}+\delta_0 \right)=\delta_{1,1}\left(12\delta_{1,1}+12\delta_{1,2}+\delta_0 \right)=0,\nonumber\\
\delta_{1,1}\left(\psi_1+\psi_2+\delta_{1,1} \right)=\psi_1\delta_2=\psi_2\delta_2=\delta_{1,1}\delta_2=0,\\
\left(\psi_1-\psi_2 \right)\left(10\psi_1+10\psi_2-2\delta_{1,1}-12\delta_{1,2}-\delta_0 \right)=0.\nonumber
\end{eqnarray}
Moreover, $R^2(\Mbar_{2,2})$ is isomorphic to $RH^2(\Mbar_{2,2})$ via the cycle map (see for instance \cite{MR2294106}). It follows that we can write 
\begin{eqnarray}
\label{basis}
\left[\overline{\mathcal{DR}}_2(d)\right] &=& 
A_{\psi_1\psi_2} \psi_1\psi_2 + A_{\psi_1^2+\psi_2^2} \left(\psi_1^2+\psi_2^2\right) + A_{\psi_1\delta_{1,1}} \psi_1\delta_{1,1} + A_{\psi_2\delta_{1,1}} \psi_2\delta_{1,1} \nonumber\\
&&{}+ A_{\psi_1\delta_{1,2}} \psi_1\delta_{1,2} + A_{\psi_2\delta_{1,2}} \psi_2\delta_{1,2} + A_{\psi_1\delta_0} \psi_1\delta_0 + A_{\psi_2\delta_0} \psi_2\delta_0 \nonumber\\
&&{}+A_{\delta_2^2} \delta_2^2  + A_{\delta_{1,2}\delta_2} \delta_{1,2}\delta_2 + A_{\delta_0\delta_2} \delta_0\delta_2 + A_{\delta_0\delta_{1,1}} \delta_0\delta_{1,1} \\
&&{}+ A_{\delta_0\delta_{1,2}} \delta_0\delta_{1,2} + A_{\delta_0^2} \delta_0^2\nonumber
\end{eqnarray}
in $R^2(\Mbar_{2,2})$.

\subsection{Test surfaces}
\label{testsurfaces}

In the following, we are going to intersect both sides of the formula in (\ref{basis}) with $10$ test surfaces in $\Mbar_{2,2}$. Each intersection will produce a linear equation in the coefficients $A$. 
{\setlength{\leftmargini}{0pt} 
\begin{enumerate}

\item \label{9} Let $C$ be a general curve of genus $2$. Consider the family of curves in $\Mbar_{2,2}$ whose fibers are obtained by varying two points in $C$. 

\begin{figure}[htbp]
\centering
  \def\svgwidth{0.4\columnwidth}
  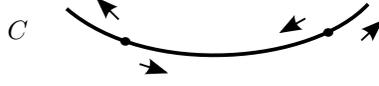
  \caption{How the general fiber of the family (1) moves.}
\end{figure}

The base of this family is $C\times C$. The family is obtained as the blow-up $\widetilde{C\times C \times C}$ of $C\times C\times C$ along the diagonal $\{(p,p,p) \,|\, p\in C \}$.

Let $\pi_i \colon C\times C\rightarrow C$ be the projection on the $i$-th factor, for $i=1,2$. Denote by $S_1$ the proper transform of $\{(p,q,p)\,|\, p,q\in C\}$, and by $S_2$ the proper transform of $\{(p,q,q)\,|\, p,q\in C\}$. On the surface $C\times C$, we have
\begin{eqnarray*}
\psi_1 &=& -c_1\left(N_{S_1/\widetilde{C\times C \times C}}\right) = \pi_1^*(K_C)+\Delta_{C\times C},\\
\psi_2 &=& -c_1\left(N_{S_2/\widetilde{C\times C \times C}}\right) = \pi_2^*(K_C)+\Delta_{C\times C},\\
\delta_2 &=& \Delta_{C\times C},
\end{eqnarray*}
and all other divisor classes restrict to zero. It follows that
\begin{eqnarray*}
\psi_i^2 &=& 2(2g(C)-2)+(2-2g(C))=2, \quad \mbox{for}\,\, i=1,2,\\
\psi_1\psi_2 &=& (2g(C)-2)(2g(C)-2) +2(2g(C)-2) +(2-2g(C))=6,\\
\delta_2^2 &=& 2-2g(C) =-2.
\end{eqnarray*}

The fibers of this family are either smooth or consist of the curve $C$ attached to a rational tail containing the two marked points. An admissible cover of one of the singular fibers has to have a ramification of order at least $2$ at the singular point. By the Riemann-Hurwitz formula, the singular fibers do not admit an admissible cover totally ramified at the marked points. A smooth fiber $(C,p,q)$ has an admissible cover of degree $d$ totally ramified at $p$ and $q$ if and only if $p-q$ is a non-trivial $d$-torsion point in $\textrm{Pic}^0(C)$. Let $\varphi\colon C\times C\rightarrow \textrm{Pic}^0(C)$ be the difference map defined as $\varphi(p,q)=\mathcal{O}_C(p-q)$. The map $\varphi$ is surjective of degree $2$ (see \cite[Chapter V]{MR770932}). It follows that $2(d^4-1)$ fibers of this family are in $\overline{\mathcal{DR}}_2(d)$. Since the admissible cover is uniquely determined, each fiber contributes with multiplicity one. We deduce the following relation
\[
4A_{\psi_1^2+\psi_2^2} + 6A_{\psi_1\psi_2} -2A_{\delta_2^2} = 2(d^4-1).
\]

\item \label{1} Let $(E_1,p_1,x_1)$ and $(E_2,p_2,x_2)$ be two $2$-pointed elliptic curves. Identify the points $p_1$ and $p_2$. Consider the family of curves whose fibers are obtained by varying $x_1\in E_1$ and $x_2\in E_2$.

\begin{figure}[htbp]
\centering
  \def\svgwidth{0.6\columnwidth}
  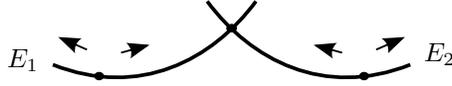
  \caption{How the general fiber of the family (2) moves.}
\end{figure}

The basis of this family is $E_1\times E_2$. To construct this family, let $\widetilde{E_i\times E_i}$ be the blow-up of $E_i\times E_i$ at $(p_i,p_i)$, let $\Gamma_i$ be the proper transform of ${p_i}\times E_i$, and let $\sigma_{\Delta_i}$ be the section corresponding to the proper transform of the diagonal $\Delta_i\subset E_i\times E_i$, for $i=1,2$. Consider $\widetilde{E_1\times E_1}\times E_2$ and $E_1\times \widetilde{E_2\times E_2}$, and identify $\Gamma_1\times E_2$ with $E_1\times\Gamma_2$. The sections $\sigma_{\Delta_1}$ and $\sigma_{\Delta_2}$ give rise to two sections of this family over $E_1\times E_2$.

Let $\pi_i\colon E_1\times E_2\rightarrow E_i$ be the projection on the $i$-th factor, for $i=1,2$. The non-zero restrictions of the divisor classes are
\begin{eqnarray*}
\psi_1 &=& -c_1\left(N_{\sigma_{\Delta_1}/\widetilde{E_1\times E_1}}\right) = \pi_1^*[p_1],\\
\psi_2 &=& -c_1\left(N_{\sigma_{\Delta_2}/\widetilde{E_2\times E_2}}\right) = \pi_2^*[p_2],\\
\delta_{1,1} &=& c_1\left(N_{\Gamma_1/\widetilde{E_1\times E_1}}\otimes N_{\Gamma_2/\widetilde{E_2\times E_2}}\right) = {}-\pi_1^*[p_1]-\pi_2^*[p_2],\\
\delta_{1,2} &=& \pi_1^*[p_1]+\pi_2^*[p_2].
\end{eqnarray*}
Hence, we obtain
\begin{align*}
 \psi_1\psi_2 &= 1, & \psi_1\delta_{1,1} = \psi_2\delta_{1,1}& = -1, & \psi_1\delta_{1,2} = \psi_2\delta_{1,2} &= 1. 
\end{align*}

A fiber of this family has an admissible cover of degree $d$ totally ramified at the two marked points if and only if $p_i-x_i$ is a non-trivial $d$-torsion point in $\textrm{Pic}^0(E_i)$, for $i=1,2$. For each such fiber $(C,x_1,x_2)$, the admissible cover $C\rightarrow D$ is unique. By the argument of \cite[Theorem 6]{MR664324}, the map from a neighborhood of the point $[C\rightarrow D]$ in the moduli space of admissible covers to the universal deformation space of $(C,x_1,x_2)$ is transverse to $\delta_{1,1}$. It follows that each fiber in the intersection of this family with $\overline{\mathcal{DR}}_2(d)$ counts with multiplicity one. We deduce
\[
A_{\psi_1\psi_2} -A_{\psi_1\delta_{1,1}} -A_{\psi_2\delta_{1,1}}+A_{\psi_1\delta_{1,2}}+A_{\psi_2\delta_{1,2}}= (d^2-1)^2.
\]

\item \label{2} Consider a general four-pointed rational curve, and attach elliptic tails varying in pencils of degree $12$ at two of the marked points.

\begin{figure}[htbp]
\centering
  \def\svgwidth{0.5\columnwidth}
  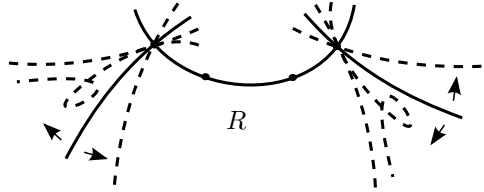
  \caption{How the general fiber of the family (3) moves.}
\end{figure}

To construct an elliptic pencil of degree $12$, blow up $\mathbb{P}^2$ in the nine points of intersection of two general cubics. Let $Y_1\rightarrow \mathbb{P}^1$ and $Y_2\rightarrow \mathbb{P}^1$ be two such elliptic pencils with zero sections $\sigma_1$ and $\sigma_2$. Let $R$ be a rational curve, and choose four general sections of $R\times\mathbb{P}^1\times\mathbb{P}^1\rightarrow \mathbb{P}^1\times\mathbb{P}^1$. Consider $Y_1\times \mathbb{P}^1$ and $\mathbb{P}^1\times Y_2$, and identify $\sigma_1\times\mathbb{P}^1$ and $\mathbb{P}^1\times\sigma_2$ with two of the sections of $R\times\mathbb{P}^1\times\mathbb{P}^1\rightarrow \mathbb{P}^1\times\mathbb{P}^1$. 

The base of the family is $\mathbb{P}^1\times\mathbb{P}^1$. Let $\pi_i\colon \mathbb{P}^1\times\mathbb{P}^1\rightarrow\mathbb{P}^1$ be the projection on the $i$-th factor, for $i=1,2$, and let $x$ be the class of a point in $\mathbb{P}^1$. The divisor classes restrict as follows
\begin{eqnarray*}
\delta_{1,2} &=& c_1\left(N_{\sigma_1/Y_1} \right) + c_1\left(N_{\sigma_2/Y_2} \right) ={}-\pi_1^*(x)-\pi_2^*(x),\\
\delta_{0} &=& {}-12 c_1\left(N_{\sigma_1/Y_1} \right) -12 c_1\left(N_{\sigma_2/Y_2} \right) =12\pi_1^*(x)+12\pi_2^*(x),
\end{eqnarray*}
and we have
\begin{align*}
\delta_0^2 &= 288, & \delta_0\delta_{1,2} &= -24.
\end{align*}
For every fiber, the two marked points are in the same rational component. An admissible cover for a fiber of this family is ramified with order at least $2$ at the two singular points. By the Riemann-Hurwitz formula, no fiber of this family admits an admissible cover totally ramified at the two marked points. We deduce the following relation
\[
288 A_{\delta_0^2} -24A_{\delta_0\delta_{1,2}} = 0.
\]

\item We consider a chain of three curves: a central rational curve with two elliptic tails. We fix a first marked point on the rational component, we vary one elliptic tail in a pencil of degree $12$, and we consider a moving second marked point on the other elliptic tail.

\begin{figure}[htbp]
\centering
  \def\svgwidth{0.5\columnwidth}
  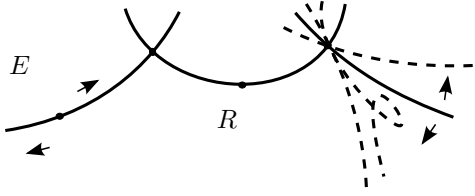
  \caption{How the general fiber of the family (4) moves.}
\end{figure}

Let $(E,p)$ be a pointed elliptic curve. Let $\widetilde{E\times E}$, $\Gamma$, and $\sigma_{\Delta}$ as in (\ref{1}). Let $Y\rightarrow \mathbb{P}^1$ be an elliptic pencil of degree $12$ as in (\ref{2}), with zero section $\sigma$. The base of the family is $E\times \mathbb{P}^1$. Consider $\widetilde{E\times E}\times \mathbb{P}^1$, $E\times Y$, and the trivial family $R\times E\times \mathbb{P}^1$ with three general sections $\tau_1$, $\tau_2$, $\tau_3$ for a rational curve $R$. Finally, identify the section $\tau_2$ with $\Gamma\times\mathbb{P}^1$, and the section $\tau_3$ with $E\times\sigma$. The sections  $\tau_1$ and $\sigma_{\Delta}$ give rise to two sections of the family over $E\times\mathbb{P}^1$. We let $\tau_1$ be the section corresponding to the first marked point, and $\sigma_{\Delta}$ the one corresponding to the second marked point.

Let $x$ be the class of a point in $\mathbb{P}^1$. The divisor classes restrict as follows
\begin{eqnarray*}
\psi_2 &=& -c_1\left(N_{\sigma_\Delta/\widetilde{E\times E}} \right)=\pi_1^*[p],\\
\delta_{1,1} &=& c_1\left( N_{\Gamma/\widetilde{E\times E}} \right) = -\pi_1^*[p],\\
\delta_{1,2} &=& \pi_1^*[p] +c_1\left( N_{\sigma/Y} \right) =\pi_1^*[p]-\pi_2^*(x),\\
\delta_0 &=& -12 c_1\left( N_{\sigma/Y} \right) = 12\pi_2^*(x),
\end{eqnarray*}
and we have
\begin{align*}
\psi_2\delta_{1,2} &= -1, & \psi_2\delta_0 =  \delta_0\delta_{1,2} &= 12, & \delta_0\delta_{1,1} &= -12.
\end{align*}
Suppose there exists an admissible cover for a fiber of this family totally ramified at the two marked points. Then, the restriction of the cover to the central rational component is totally ramified at the marked point and at one of the singular points, plus it has at least a simple ramification at the other singular point, a contradiction. 
Hence, this family has empty intersection with $\overline{\mathcal{DR}}_2(d)$, and we obtain
\[
{}-A_{\psi_2\delta_{1,2}}+12A_{\psi_2\delta_0} -12A_{\delta_0\delta_{1,1}} +12A_{\delta_0\delta_{1,2}}=0.
\]

\item \label{11} Pick a general element in the divisor $\Delta_{1,1}$, and consider the surface obtained by varying the first marked point and the $j$-invariant of the elliptic tail on which it lies.

\begin{figure}[htbp]
\centering
  \def\svgwidth{0.5\columnwidth}
  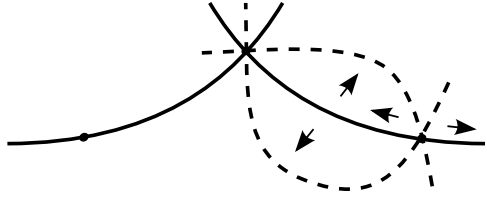
  \caption{How the general fiber of the family (5) moves.}
\end{figure}

The base of the family is the blow-up of $\mathbb{P}^2$ in the nine points of intersection of two general cubics. Let $H$ be the pullback of the hyperplane class in $\mathbb{P}^2$, let $\Sigma$ denote the sum of the classes of the nine exceptional divisors, and let $E_0$ be the class of one of the exceptional divisors.

The divisor classes of $\Mbar_{2,2}$ restrict as follows
\begin{eqnarray*}
\psi_1 &=& 3H-\Sigma+E_0,\\
\delta_{1,1} &=& {}-3H+\Sigma-E_0,\\
\delta_{1,2} &=& E_0,\\
\delta_0 &=& 36H-12\Sigma,
\end{eqnarray*}
and we obtain
\begin{align*}
\psi_1^2 &= 1, & \psi_1\delta_{1,1} &=-1, & \psi_1\delta_0 = \delta_0\delta_{1,2} &=12, & \delta_0\delta_{1,1} &= -12.
\end{align*}
Since the second marked point is general, the surface is disjoint from the locus $\overline{\mathcal{DR}}_2(d)$, hence we deduce the following relation
\[
 A_{\psi_1^2+\psi_2^2} - A_{\psi_1\delta_{1,1}} +12A_{\psi_1\delta_0}-12A_{\delta_0\delta_{1,1}}+12A_{\delta_0\delta_{1,2}}=0.
\]

\item Choose a general element in 
the locus of curves with a rational and an elliptic component meeting at two non-disconnecting nodes, with both marked points on the rational component, 
and consider the surface obtained by varying the two moduli of the two-pointed elliptic component.

\begin{figure}[htbp]
\centering
  \def\svgwidth{0.3\columnwidth}
  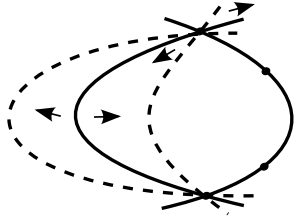
  \caption{How the general fiber of the family (6) moves.}
\end{figure}

The base of this family is the same surface as that of the family (\ref{11}). The restriction of the divisor classes are
\begin{eqnarray*}
 \delta_{1,2} &=& E_0,\\
\delta_0 &=& 30 H-10\Sigma-2E_0,
\end{eqnarray*}
(see also \cite[\S 2 (9)]{MR1023390}), and we have
\begin{align*}
\delta_0\delta_{1,2} &=12,  & \delta_0^2 &= -44.
\end{align*}
Since we can take the two singular points to be general in the rational component, this family is disjoint from $\overline{\mathcal{DR}}_2(d)$. We deduce the following relation
\[
 12A_{\delta_0\delta_{1,2}} -44 A_{\delta_0^2}=0.
\]

\item \label{8} Pick a general element in the 
intersection of the two divisors $\delta_{1,2}$ and $\delta_2$, 
and vary the two moduli of the central two-pointed elliptic component.
 Once again, the base of this family is the same surface as that of the family (\ref{11}).
 
 \begin{figure}[htbp]
\centering
  \def\svgwidth{0.5\columnwidth}
  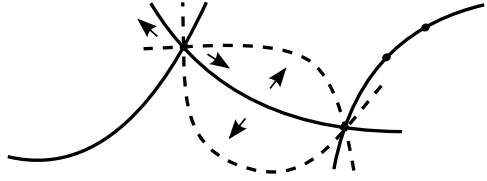
  \caption{How the general fiber of the family (7) moves.}
\end{figure}

Suppose there exists an admissible cover for a fiber of this family, totally ramified at the two marked points. Then, the restriction of the cover to the rational component is totally ramified at the two marked points and has at least a simple ramification at the singular point, a contradiction.
 It follows that the family is disjoint from $\overline{\mathcal{DR}}_2(d)$. The divisor classes restrict as follows
\begin{eqnarray*}
 \delta_2 &=& {}-3H +\Sigma -E_0,\\
 \delta_{1,2} &=& {}-3H +\Sigma, \\
 \delta_0 &=& 36 H-12\Sigma
\end{eqnarray*}
(see also \cite[\S 3 ($\lambda$)]{MR1078265}). Hence, we obtain
\begin{align*}
\delta_2^2 = \delta_{1,2}\delta_2 &=1, & \delta_0\delta_2 &=-12,
\end{align*}
and
\[
 A_{\delta_2^2}+A_{\delta_{1,2}\delta_2} -12A_{\delta_0\delta_2}=0.
\]

\item Let $(E,p)$ be a pointed elliptic curve. Attach a rational tail with the two marked points at $p$. Identify a moving point in $E$ with a base point of a pencil of elliptic curves of degree $12$.

\begin{figure}[htbp]
\centering
  \def\svgwidth{0.5\columnwidth}
  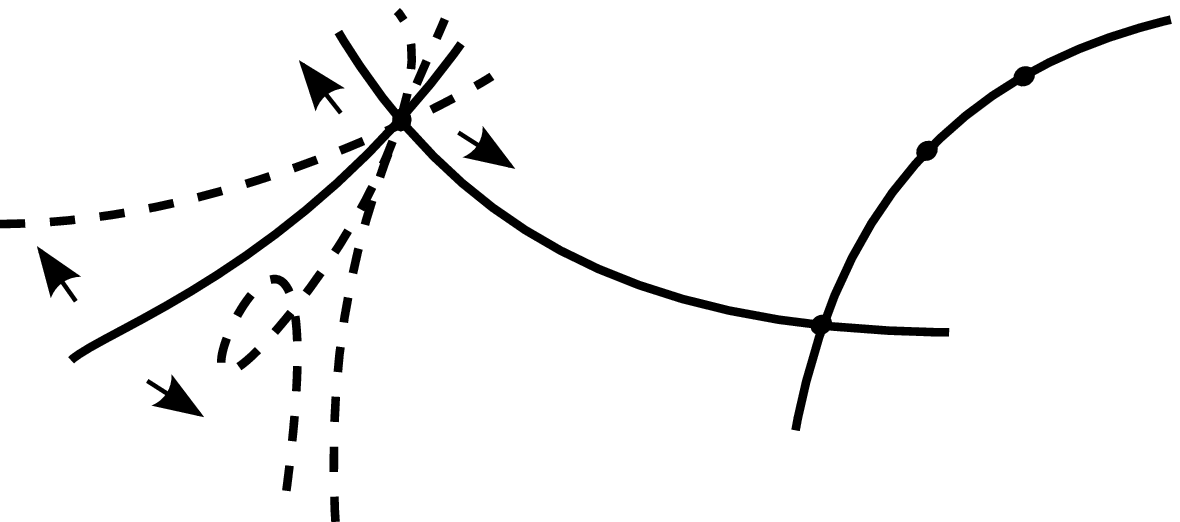
  \caption{How the general fiber of the family (8) moves.}
\end{figure}

Let $Y\rightarrow \mathbb{P}^1$ be an elliptic pencil of degree $12$ as in (\ref{2}), with zero section $\sigma$.
Let $\widetilde{E\times E}$, $\Gamma$, and $\sigma_{\Delta}$ as in (\ref{1}). 
The base of this family is $\mathbb{P}^1\times E$. 
Let $x$ be the class of a point in $\mathbb{P}^1$.
The divisor classes are
\begin{eqnarray*}
 \delta_2 &=& c_1\left(N_{\Gamma/\widetilde{E\times E}} \right)= -\pi_2^*[p],\\
\delta_{1,2} &=& c_1\left(N_{\sigma/Y}\otimes N_{\sigma_{\Delta}/\widetilde{E\times E}} \right) + \Gamma = -\pi_1^*(x),\\
\delta_0 &=& 12 \pi_1^*(x).
\end{eqnarray*}
It follows that
\begin{align*}
\delta_{1,2}\delta_2 &= 1, & \delta_0\delta_2 &= -12.
\end{align*}
As for the family (\ref{8}), this family does not meet $\overline{\mathcal{DR}}_2(d)$, hence we obtain
\[
 A_{\delta_{1,2}\delta_2}-12A_{\delta_0\delta_2}=0.
\]

\item Consider a $4$-pointed rational curve $(R,p_1,p_2,s,t)$. Identify the point $s$ with the base point of an elliptic pencil, and identify the point $t$ with a moving point in $R$.

\begin{figure}[htbp]
\centering
  \def\svgwidth{0.5\columnwidth}
  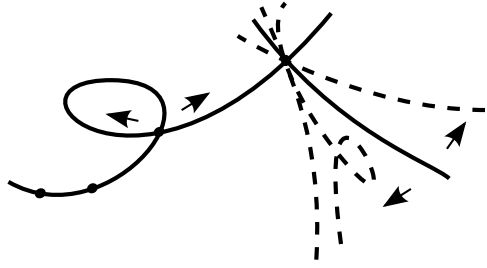
  \caption{How the general fiber of the family (9) moves.}
\end{figure}

The base of the family is $R\times \mathbb{P}^1$. To construct the surface, let $\widetilde{R\times R}$ be the blow-up of $R\times R$ at the points $(p_1,p_1)$, $(p_2,p_2)$, $(s,s)$, $(t,t)$. Let $E_{p_1}$, $E_{p_2}$, $E_s$, $E_t$, and $E_\Delta$ be the proper transform respectively of $\{p_1\}\times R$, $\{p_2\}\times R$, $\{s\}\times R$, $\{t\}\times R$, and $\Delta\subset R\times R$. Let $Y\rightarrow\mathbb{P}^1$ be as in the family (\ref{1}), and $\sigma$ be the zero section. Consider $\widetilde{R\times R}\times \mathbb{P}^1$ and $R\times Y$, and identify  $E_s\times\mathbb{P}^1$ with $R\times\sigma$, and $E_\Delta\times\mathbb{P}^1$ with $E_t\times \mathbb{P}^1$.

By an argument similar to the one used for the family (\ref{8}), this family does not meet the locus $\overline{\mathcal{DR}}_2(d)$. The divisor classes restrict as follows
\begin{eqnarray*}
 \psi_1 &=& -c_1\left(N_{E_{p_1}/\widetilde{R\times R}} \right) = \pi_1^*[p_1],\\
 \psi_2 &=& -c_1\left(N_{E_{p_2}/\widetilde{R\times R}} \right) = \pi_1^*[p_2],\\
 \delta_{1,2} &=& E_t + c_1\left(N_{E_s/\widetilde{R\times R}} \otimes N_{\sigma/Y} \right)= \pi_1^*(t) -\pi_1^*(s) -\pi_2^*(x),\\
 \delta_0 &=& c_1\left(N_{E_t/\widetilde{R\times R}}\otimes N_{E_\Delta/\widetilde{R\times R}} \right)+E_{p_1}+E_{p_2}+E_s -12 c_1\left(N_{\sigma/Y} \right)= 12\pi_2^*(x).
\end{eqnarray*}
We deduce that
\begin{align*}
\psi_1\delta_{1,2} = \psi_2\delta_{1,2} &= -1, & \psi_1\delta_0 = \psi_2\delta_0 &=12,
\end{align*}
and
\[
 {}-A_{\psi_1\delta_{1,2}}+12A_{\psi_1\delta_0}-A_{\psi_2\delta_{1,2}}+12A_{\psi_2\delta_0}=0.
\]

\item Let $(R, q_1, q_2, q_3, q_4, q_5)$ be a $5$-pointed rational curve. Attach a rational curve with two marked points at $q_1$. Finally, identify $q_2$ with $q_3$, and $q_4$ with $q_5$. Consider the surface obtained by varying the two moduli of the $5$-pointed rational curve.

\begin{figure}[htbp]
\centering
  \def\svgwidth{0.5\columnwidth}
  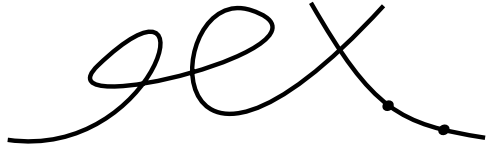
  \caption{The general fiber of the family (10).}
\end{figure}

The base of this surface is $\Mbar_{0,5}$. 
For $i,j\in\{1,2,3,4,5\}$, let $D_{i,j}$ be the class of the divisor in $\Mbar_{0,5}$ corresponding to singular rational curves with two components, the markings $i,j$ in one component and the other three markings in the other component. Note that $D_{i,j}^2=-1$, and $D_{i,j}\cdot D_{k,l}=1$ if $\{i,j\}\cap\{k,l\}=\emptyset$, otherwise $D_{i,j}\cdot D_{k,l}=0$.

Let $\widetilde{\psi}_i$ be the cotangent line class in $\Mbar_{0,5}$ corresponding to the marking $i$. If we fix two markings $j,k$ different from $i$, and if $l,m$ are such that $\{1,2,3,4,5\}=\{i,j,k,l,m\}$, then note that we can write
\[
 \widetilde{\psi}_i = D_{j,k}+D_{i,l}+D_{i,m}.
\]
The restrictions of the divisor classes in $\Mbar_{2,2}$ to this family are
\begin{eqnarray*}
 \delta_2 &=& -\widetilde{\psi}_1\\
	&=& {}-D_{2,3}-D_{1,5}-D_{1,4},\\
 \delta_{1,2} &=& D_{4,5}+D_{2,3},\\
 \delta_0 &=& {}-\widetilde{\psi}_2-\widetilde{\psi}_3-\widetilde{\psi}_4-\widetilde{\psi}_5\\
&& {}+D_{1,2}+D_{1,3}+D_{1,4}+D_{1,5}\\
&& {}+D_{2,4}+D_{2,5}+D_{3,4}+D_{3,5}\\
&=& {}-2D_{4,5}-D_{2,3}-D_{3,4}-D_{1,2}+D_{1,4}.
\end{eqnarray*}
It follows that
\begin{align*}
\delta_2^2 &=1, & \delta_{1,2}\delta_2 &= -2, & \delta_0\delta_2 &= 4.
\end{align*}
As for the family (\ref{8}), this family is disjoint from the locus $\overline{\mathcal{DR}}_2(d)$, and we obtain
\[
 A_{\delta_2^2}-2A_{\delta_{1,2}\delta_2}+4A_{\delta_0\delta_2}=0.
\]

\end{enumerate} }

\subsection{Symmetry}
\label{symmetry} 
The locus $\mathcal{DR}_2(d)$ is symmetric in the two marked points. We deduce that, for a choice of basis of $R^2(\Mbar_{2,2})$ symmetric with respect to the classes $\psi_1$, $\psi_2$, the expression of the class of $\overline{\mathcal{DR}_2(d)}$ is symmetric in $\psi_1$, $\psi_2$. Hence, we have the following three relations
\begin{eqnarray*}
A_{\psi_1\delta_{1,1}}&=&A_{\psi_2\delta_{1,1}},\\
A_{\psi_1\delta_{1,2}}&=&A_{\psi_2\delta_{1,2}},\\
A_{\psi_1\delta_{0}}&=&A_{\psi_2\delta_{0}}.
\end{eqnarray*}

\subsection{The push-forward to $\Mbar_{2,1}$}
\label{push}
  It is straightforward to compute the pushforward of products of divisor classes via the maps $\pi_i$: when $i=1$, we have
\begin{align*}
 &(\pi_1)_*\left(\psi_1^2 \right) = \kappa_1 = \psi+\frac{1}{5}\delta_0 +\frac{7}{5}\delta_1, &&(\pi_1)_*\left(\psi_2^2 \right) = \psi, && (\pi_1)_*\left(\psi_1\psi_2 \right) = 3\psi,\\
&(\pi_1)_*\left(\psi_1\delta_{1,1} \right) = \delta_1, &&(\pi_1)_*\left(\psi_1\delta_{1,2} \right) = 2\delta_1, && (\pi_1)_*\left(\psi_2\delta_{1,2} \right) = \delta_1,\\
&(\pi_1)_*\left(\psi_1\delta_0 \right) = 3\delta_0, &&(\pi_1)_*\left(\psi_2\delta_0 \right) = \delta_0, && (\pi_1)_*\left(\delta_2^2 \right) = -\psi,\\
&(\pi_1)_*\left(\delta_{1,2}\delta_2 \right) = \delta_1, &&(\pi_1)_*\left(\delta_0\delta_2 \right) = \delta_0,
\end{align*}
and all other products have zero pushforward. It follows that the pushforward of the expression in (\ref{basis}) is
\begin{eqnarray*}
 (\pi_1)_*\left[ \overline{\mathcal{DR}}_2(d) \right] &=&  \left(3A_{\psi_1\psi_2}+2A_{\psi_1^2+\psi_2^2}-A_{\delta_2^2} \right)\psi \\
   && {}+ \left(\frac{1}{5}A_{\psi_1^2+\psi_2^2} +3A_{\psi_1\delta_0}+A_{\psi_2\delta_0}+A_{\delta_0\delta_2} \right) \delta_0 \\
   && {}+ \left(\frac{7}{5}A_{\psi_1^2+\psi_2^2}+A_{\psi_1\delta_{1,1}}+2A_{\psi_1\delta_{1,2}}+A_{\psi_2\delta_{1,2}}+A_{\delta_{1,2}\delta_2} \right) \delta_1.
\end{eqnarray*}
Comparing with Corollary \ref{pfwdclass}, we deduce the following three relations
\begin{eqnarray*}
 3A_{\psi_1\psi_2}+2A_{\psi_1^2+\psi_2^2}-A_{\delta_2^2} &=& d^4-1,\\
\frac{1}{5}A_{\psi_1^2+\psi_2^2} +3A_{\psi_1\delta_0}+A_{\psi_2\delta_0}+A_{\delta_0\delta_2} &=& -\frac{(d^2-1)(d^2+6)}{60},\\
\frac{7}{5}A_{\psi_1^2+\psi_2^2}+A_{\psi_1\delta_{1,1}}+2A_{\psi_1\delta_{1,2}}+A_{\psi_2\delta_{1,2}}+A_{\delta_{1,2}\delta_2} &=& -\frac{(d^2-1)(d^2+6)}{5}.
\end{eqnarray*}

\subsection{Proof of Theorem \ref{ACC}}
In order to determine the class of $\overline{\mathcal{DR}}_2(d)$, we need to determine the coefficients $A$ in formula (\ref{basis}). In \S\ref{testsurfaces} we have found $10$ relations using test surfaces. In \S\ref{symmetry} we have $3$ relations coming from the symmetry of  $\overline{\mathcal{DR}}_2(d)$ in the two marked points. Finally, in \S\ref{push} we have deduced $3$ more relations from the study of the push-forward of $\overline{\mathcal{DR}}_2(d)$ to $\Mbar_{2,1}$. In total, we have a system of $16$ linear relations in the coefficients $A$. The associated matrix has rank $14$, and, solving the system, we prove the statement. The linear system is consistent, hence there are two  redundant relations which serve as checks on our computation. \hfill$\Box$

\subsection{Test}
Although classes of closures of double ramification cycles are not known in general, in \cite{BSSZ} all zero-dimensional intersections in $\Mbar_{g,n}$ of double ramification cycles with monomials in $\psi$-classes have been computed. It is interesting to note that the authors show that for $\underline{d}=(d_1,\dots,d_n)$ with all $d_i$ non-zero, the closures of double ramification loci $\mathcal{DR}_g(\underline{d})$ by means of admissible covers and the push-forward of the virtual fundamental class of the space of relative stable maps to an unparametrized $\mathbb{P}^1$ have same intersections with monomials in $\psi$-classes.

For instance, they show that
\begin{eqnarray}
 \label{psiint}
 \left[\overline{\mathcal{DR}}_2(d) \right]\cdot \psi_1^3 = \left[\overline{\mathcal{DR}}_2(d) \right]\cdot \psi_2^3 = \frac{(d^2-1)(3d^2-7)}{5760}. 
\end{eqnarray}
Using the explicit espression of the class of $\overline{\mathcal{DR}}_2(d)$ from Theorem \ref{ACC}, we verify the intersections in (\ref{psiint}) as a 
check of our computation:
\begin{eqnarray}
 \left[\overline{\mathcal{DR}}_2(d) \right]\cdot \psi_1^3 &=& \left[\overline{\mathcal{DR}}_2(d) \right]\cdot (\pi_2^*\psi +\delta_2)^3 \nonumber\\
	&=& \left[\overline{\mathcal{DR}}_2(d) \right]\cdot \pi_2^*\psi^3 \label{intdelta2}\\
	&=& (\pi_2)_*\left[\overline{\mathcal{DR}}_2(d) \right]\cdot \psi^3 \nonumber\\
	&=& (d^2-1)\left((d^2+1)\psi-\frac{d^2+6}{5}\left(\frac{1}{12}\delta_0+\delta_1\right) \right)\cdot\psi^3 \nonumber\\
	&=& \frac{(d^2-1)(3d^2-7)}{5760}. \label{Faberint}
\end{eqnarray}
Note that in (\ref{intdelta2}) we have used that $\psi_i\cdot\delta_2=0$ in $\Mbar_{2,2}$, and in (\ref{Faberint}) we have used the following intersections in $\Mbar_{2,1}$
\begin{eqnarray*}
 \psi^4 = \frac{1}{1152},\quad\quad \psi^3\delta_0 = \frac{1}{48},\quad\quad \psi^3\delta_1 = 0
\end{eqnarray*}
computed in \cite[Chapter 3]{faberthesis}.

\section{Comparing with the pull-back of the zero section}
\label{Hain}

The class in $A^g(\M_{g,n}^{ct})$ of the pull-back of the zero section ${\mathcal{Z}}^{ct}_g$ of the universal Jacobian variety $\mathcal{J}^{ct}_{g,n}\rightarrow\M_{g,n}^{ct}$ has been computed in \cite{Hain} and \cite{GZ1}. The space $R^2(\mathcal{M}^{ct}_{2,2})$ has dimension $5$. We can consider the basis formed by degree-two monomials in the divisor classes $\psi_1, \psi_2, \delta_{1,1}, \delta_{1,2}, \delta_2$ modulo the restrictions to $\M_{2,2}^{ct}$ of the relations (\ref{relations}) in $\Mbar_{2,2}$ and the following additional relations that hold in $\M_{2,2}^{ct}$
\begin{eqnarray}
\label{relations2}
 \psi_1\psi_2 = \frac{3}{2}\left(\psi_1^2 + \psi_2^2\right) -\frac{9}{10} \left(\psi_1 + \psi_2 \right)\delta_{1,1} -\frac{2}{5} \left(\psi_1 + \psi_2\right)\delta_{1,2} \nonumber\\
 \psi_i^2 = \frac{7}{10} \left(\psi_i \left(\delta_{1,1}+\delta_{1,2} \right) - \delta_{1,2}\delta_2\right) -\delta_2^2\\
 \left(\psi_1-\psi_2\right) \delta_{1,1} = \left(\psi_1-\psi_2\right) \delta_{1,2}.\nonumber
\end{eqnarray}
To prove the above relations, it is enough to express the product of divisor classes in terms of decorated boundary strata classes as in \cite{MR1672112}.
To have a basis symmetric with respect to the two marked points, we choose the basis given by the following classes:
\[
 \left(\psi_1+\psi_2\right)\delta_{1,1},\quad \psi_1\delta_{1,2},\quad \psi_2\delta_{1,2},\quad \delta_2^2,\quad \delta_{1,2}\delta_2.
\]

\begin{proof}[Proof of Corollary \ref{Hain-AC}]
From \cite{Hain} and \cite{GZ1} the class of $\varphi^*_{(d,-d)}{\mathcal{Z}}^{ct}_2$ in $A^2(\M_{2,2}^{ct})$ can be expressed as follows
\[
 \left[\varphi^*_{(d,-d)}{\mathcal{Z}}^{ct}_2\right] = \frac{1}{2}\left(\frac{d^2}{2}\left( (\psi_1-\delta_2)+(\psi_2-\delta_2) \right)+d^2\delta_2-\frac{d^2}{2}\delta_{1,1} \right)^2.
\]
Using the relations in (\ref{relations}) and (\ref{relations2}), we have
\[
 \left[\varphi^*_{(d,-d)}{\mathcal{Z}}^{ct}_2\right] = d^4 \left(\frac{1}{4}\left(\psi_1+\psi_2\right)\left(\delta_{12}-\delta_{11}\right)-\delta_2^2-\frac{7}{10}\delta_{12}\delta_2 \right),
\]
and the restriction of the class of $\overline{\mathcal{DR}}_2(d)$ from $\Mbar_{2,2}$ is
\[
 \left[\overline{\mathcal{DR}}_2(d)\right]= (d^2-1)\left(\frac{d^2-1}{4}\left(\psi_1+\psi_2 \right)\left(\delta_{1,2}-\delta_{1,1} \right)-(d^2+1)\delta_2\left(\delta_2+\frac{7}{10}\delta_{1,2}\right) \right).
\]
It follows that
\[
 \left[\varphi^*_{(d,-d)}{\mathcal{Z}}^{ct}_2\right] - \left[\overline{\mathcal{DR}}_2(d)\right] = \frac{2d^2-1}{4}\left(\psi_1+\psi_2 \right)\left(\delta_{1,2}-\delta_{1,1} \right)-\delta_2\left(\delta_2+\frac{7}{10}\delta_{1,2}\right).
\]
In \cite{MR1672112} Getzler describes also a basis for $R^2(\Mbar_{2,2})$ made of decorated boundary strata classes and explains the change of basis. 
In terms of decorated boundary strata classes, using the notation in \cite{MR1672112} (see the introduction), we have
\begin{eqnarray}
\label{DRct}
 \left[\varphi^*_{(d,-d)}{\mathcal{Z}}^{ct}_2\right] &=& d^4 \left(\delta_{22}+\delta_{11|}-\frac{1}{5}\delta_{11|12} \right)\nonumber\\
 \left[\overline{\mathcal{DR}}_2(d)\right] &=& (d^2-1)\left((d^2+1)\delta_{22}+(d^2-1)\delta_{11|}-\frac{d^2+6}{5}\delta_{11|12} \right)
\end{eqnarray}
in $R^2(\M_{2,2}^{ct})$, hence the statement. 
\end{proof}

\section{In the cone of effective classes in codimension two}
\label{codim2cone}

In this section we study the position of the classes of the loci $\overline{\mathcal{DR}}_2(d)$ inside the cone in $A^2(\Mbar_{2,2})$ spanned by classes of effective codimension-two loci in $\Mbar_{2,2}$.

The first result is Corollary \ref{CI}, that is, the classes of the loci $\overline{\mathcal{DR}}_2(d)$ are outside the cone of classes of complete intersections.

\begin{proof}[Proof of Corollary \ref{CI}]
Suppose that the class of $\overline{\mathcal{DR}}_2(d)$ is $\left[\overline{D}_1 \right]\cdot \left[\overline{D}_2 \right]$, where $D_1, D_2$ are two effective divisors in $\M_{2,2}$. For $i=1,2$, the class $\left[\overline{D}_i \right]$ can be expressed as
\[
 \left[\overline{D}_i \right] = c^{(i)}_{\psi_1} \psi_1 + c^{(i)}_{\psi_2} \psi_2 - c^{(i)}_{\delta_0} \delta_0 - c^{(i)}_{\delta_2}\delta_2 - c^{(i)}_{\delta_{1,1}}\delta_{1,1} - c^{(i)}_{\delta_{1,2}}\delta_{1,2} \in
\textrm{Pic}(\Mbar_{2,2}),
\]
where the coefficients $c^{(i)}$ are non-negative. Using the relations (\ref{relations}), the product $\left[\overline{D}_1 \right]\cdot \left[\overline{D}_2 \right]$ can be expressed in terms of the basis chosen in (\ref{basis}). The coefficient of the class $\psi_1^2+\psi_2^2$ is 
\[
\frac{1}{2}\left(c^{(1)}_{\psi_1}c^{(2)}_{\psi_1}+c^{(1)}_{\psi_2}c^{(2)}_{\psi_2} \right)
\]
hence non-negative, a contradiction.
\end{proof}

An immediate consequence of Theorem \ref{ACC} is Corollary \ref{cone}, that is, the classes of the loci $\overline{\mathcal{DR}}_2(d)$ lie in the two-dimensional cone spanned by the classes 
$\left[\overline{\mathcal{DR}}_2(2)\right]$ and $\left[\overline{\mathcal{DR}}_2(\infty)\right]$.

\begin{proof}[Proof of Corollary \ref{cone}]
It is easy to verify that
\begin{eqnarray}
\label{2dimcone}
 \left[\overline{\mathcal{DR}}_2(d)\right] = (d^2-1)\left(\frac{1}{3}\cdot\left[\overline{\mathcal{DR}}_2(2)\right] + (d^2-4)\cdot\left[\overline{\mathcal{DR}}_2(\infty)\right] \right),
\end{eqnarray}
hence the statement.
\end{proof}

In \S\ref{div} we have studied the cone of the classes $(\pi_i)_*\left(\overline{\mathcal{DR}}_2(d)\right)$ in $\Mbar_{2,1}$. In particular, we have seen that the class of the push-forward of $\overline{\mathcal{DR}}_2(2)$ lies in an extremal ray of the cone of effective divisor classes of $\Mbar_{2,1}$, while for $d\geq 3$ the class of the push-forward of $\overline{\mathcal{DR}}_2(d)$ lies in the interior of the cone of effective divisor classes.
It is natural to ask whether the class $\overline{\mathcal{DR}}_2(2)$ lie in an extremal ray of the cone of effective codimension-two classes of $\Mbar_{2,2}$. Here we prove that all classes $\overline{\mathcal{DR}}_2(d)$ with $d\geq 3$ are not extremal.
From Corollary \ref{cone}, it is enough to show that the class $\left[\overline{\mathcal{DR}}_2(\infty)\right]$ is not extremal.

Besides the basis of $R^2(\Mbar_{2,2})$ given by the products of divisor classes in (\ref{basis}), in \cite{MR1672112} Getzler describes also a basis of $R^2(\Mbar_{2,2})$ made of the decorated class $\delta_{22}$ (defined after Corollary \ref{Hain-AC}) together with boundary strata classes, and explains the change of basis. 
From (\ref{DRct}), we see that the coefficient of $\delta_{22}$ in the expression of $\left[\overline{\mathcal{DR}}_2(2)\right]$ in terms of decorated boundary strata classes is non-zero. Hence, we can consider the basis of $R^2(\Mbar_{2,2})$ made of the class $\left[\overline{\mathcal{DR}}_2(2)\right]$
together with boundary strata classes. The expression of the class $\left[\overline{\mathcal{DR}}_2(\infty)\right]$ in terms of such a basis is
\begin{eqnarray*}
\left[\overline{\mathcal{DR}}_2(\infty)\right] = \frac{1}{5} \delta_{1,2}\delta_2 + \frac{1}{60}\delta_0\delta_2 +\frac{2}{5} \delta_{11|} +\frac{1}{30} (\delta_{01|} + \delta_{0|}) + \frac{1}{360} \delta_{00} + \frac{1}{15} \left[\overline{\mathcal{DR}}_2(2)\right].
\end{eqnarray*}
Note that the product $\delta_{1,2}\delta_2$ corresponds to the class of the locus of curves with a central elliptic component, an elliptic tail, and a rational tail with both marked points. The product $\delta_{0}\delta_2$ is the class of the locus of curves with a nodal component of geometric genus $1$ attached at a rational component with the two marked points. The class $\delta_{11|}$ is defined immediately after Corollary \ref{Hain-AC}. 
Moreover, $\delta_{01|}$ is the class of the locus of curves obtained by attaching an elliptic tail at a rational curve with two marked points and an irreducible node; $\delta_{0|}$ is the class of the locus of curves with an elliptic component meeting a rational component in two points, and both marked points on the rational component; finally, $\delta_{00}$ is the class of the locus of curves with two non-disconnecting nodes.

The above identity expresses the class $\left[\overline{\mathcal{DR}}_2(\infty)\right]$ as a linear combination with positive coefficients of the effective codimension-two classes on the right-hand side.
This proves the following statement.

\begin{prop}
For $d\geq 3$, the class $\left[\overline{\mathcal{DR}}_2(d)\right]$ is not extremal in the cone of effective codimension-two classes of $\Mbar_{2,2}$.
\end{prop}

\appendix

\section{\texorpdfstring{Non-polynomiality for $n\geq 3$}{Non-polynomiality}}

By varying the multi-index $\underline{d}=(d_1,\dots,d_n)\in\mathbb{Z}^n$, 
the class in $A^g(\M_{g,n}^{ct})$ of the pull-back of the zero section ${\mathcal{Z}}_g^{ct}$ of the universal Jacobian $\mathcal{J}^{ct}_g\rightarrow \M_{g,n}^{ct}$ can be viewed as a polynomial in the indices $d_i$ (\cite{Hain}). The push-forward of the virtual fundamental class of the space of relative stable maps to an unparametrized $\mathbb{P}^1$ coincides with the class of the pull-back of ${\mathcal{Z}}_g^{ct}$ in $A^g(\M_{g,n}^{ct})$ (\cite{MR2864866}), and similarly its extension to $A^g(\Mbar_{g,n})$ is conjectured to be a polynomial in $d_i$.

On the other hand, the closures of the double ramification loci computed by means of admissible covers are not expected to be polynomials in the indices $d_i$ (\cite{BSSZ}). In this section, we verify this expectation in the case $g=2$, $n\geq 3$ by the following easy computation.

Consider the following surface in $\Mbar_{2,n}$, for $n\geq 3$. Attach a rational curve containing the marked points $1,\dots,n-1$ to a moving point on a general curve $C$ of genus $2$, and choose the point $n$ to be a moving point in $C$. 
The construction in (\ref{9}) of \S \ref{adm} yields a family over this surface.

Let us study the intersection of this surface with the locus $\overline{\mathcal{DR}}_2(\underline{d})$ in $\Mbar_{2,n}$, for some multi-index $\underline{d}=(d_1,\dots,d_n)$ of degree $0$. If $d_n\not= 0$, the intersection is analogous to the one studied for the family (\ref{9}) in \S \ref{adm}, that is, $2(d_n^2-1)$. On the other hand, the intersection is empty if $d_n=0$. Since the result of this intersection is not a polynomial in $d_n$, we deduce the non-polynomiality of the class of $\overline{\mathcal{DR}}_2(\underline{d})$ in the indices $d_i$.

\bibliographystyle{abbrv}
\bibliography{Biblio.bib}

\begin{thebibliography}{10}

\bibitem{MR626954}
E.~Arbarello and M.~Cornalba.
\newblock Footnotes to a paper of {B}eniamino {S}egre: ``{O}n the modules of
  polygonal curves and on a complement to the {R}iemann existence theorem''
  ({I}talian) [{M}ath. {A}nn. {100} (1928), 537--551;\ {J}buch {54}, 685].
\newblock {\em Math. Ann.}, 256(3):341--362, 1981.

\bibitem{MR770932}
E.~Arbarello, M.~Cornalba, P.~A. Griffiths, and J.~Harris.
\newblock {\em Geometry of algebraic curves. {V}ol. {I}}, volume 267 of {\em
  Grundlehren der Mathematischen Wissenschaften}.
\newblock Springer-Verlag, New York, 1985.

\bibitem{MR2294106}
D.~Arcara and Y.-P. Lee.
\newblock Tautological equations in genus 2 via invariance constraints.
\newblock {\em Bull. Inst. Math. Acad. Sin. (N.S.)}, 2(1):1--27, 2007.

\bibitem{BSSZ}
A.~Buryak, S.~Shadrin, L.~Spitz, and D.~Zvonkine.
\newblock Integrals of {$\psi$}-classes over double ramification cycles.
\newblock {\em Preprint, arXiv:1211.5273}, 2012.

\bibitem{MR2864866}
R.~Cavalieri, S.~Marcus, and J.~Wise.
\newblock Polynomial families of tautological classes on {$\M_{g,n}^{rt}$}.
\newblock {\em J. Pure Appl. Algebra}, 216(4):950--981, 2012.

\bibitem{MR3231020}
D.~Chen and I.~Coskun.
\newblock Extremal effective divisors on {$\overline{\M}_{1,n}$}.
\newblock {\em Math. Ann.}, 359(3-4):891--908, 2014.

\bibitem{MR791679}
S.~Diaz.
\newblock Exceptional {W}eierstrass points and the divisor on moduli space that
  they define.
\newblock {\em Mem. Amer. Math. Soc.}, 56(327):iv+69, 1985.

\bibitem{MR910206}
D.~Eisenbud and J.~Harris.
\newblock The {K}odaira dimension of the moduli space of curves of genus {$\geq
  23$}.
\newblock {\em Invent. Math.}, 90(2):359--387, 1987.

\bibitem{faberthesis}
C.~Faber.
\newblock {\em Chow rings of moduli spaces of curves}.
\newblock PhD thesis, Amsterdam, 1988.

\bibitem{MR1023390}
C.~Faber.
\newblock Some results on the codimension-two {C}how group of the moduli space
  of stable curves.
\newblock In {\em Algebraic curves and projective geometry ({T}rento, 1988)},
  volume 1389 of {\em Lecture Notes in Math.}, pages 66--75. Springer, Berlin,
  1989.

\bibitem{MR1078265}
C.~Faber.
\newblock Chow rings of moduli spaces of curves. {II}. {S}ome results on the
  {C}how ring of {$\overline{\mathcal M}_4$}.
\newblock {\em Ann. of Math. (2)}, 132(3):421--449, 1990.

\bibitem{MR2120989}
C.~Faber and R.~Pandharipande.
\newblock Relative maps and tautological classes.
\newblock {\em J. Eur. Math. Soc. (JEMS)}, 7(1):13--49, 2005.

\bibitem{MR2722511}
C.~Faber, S.~Shadrin, and D.~Zvonkine.
\newblock Tautological relations and the {$r$}-spin {W}itten conjecture.
\newblock {\em Ann. Sci. \'Ec. Norm. Sup\'er. (4)}, 43(4):621--658, 2010.

\bibitem{MR2574363}
G.~Farkas.
\newblock The {F}ermat cubic and special {H}urwitz loci in {$\overline{\M}_g$}.
\newblock {\em Bull. Belg. Math. Soc. Simon Stevin}, 16(5, Linear systems and
  subschemes):831--851, 2009.

\bibitem{MR1672112}
E.~Getzler.
\newblock Topological recursion relations in genus {$2$}.
\newblock In {\em Integrable systems and algebraic geometry ({K}obe/{K}yoto,
  1997)}, pages 73--106. World Sci. Publ., River Edge, NJ, 1998.

\bibitem{GZ1}
S.~Grushevsky and D.~Zakharov.
\newblock The double ramification cycle and the theta divisor.
\newblock {\em Proc. Amer. Math. Soc.}, 142(12):4053--4064, 2014.

\bibitem{MR3189435}
S.~Grushevsky and D.~Zakharov.
\newblock The zero section of the universal semiabelian variety and the double
  ramification cycle.
\newblock {\em Duke Math. J.}, 163(5):953--982, 2014.

\bibitem{Hain}
R.~Hain.
\newblock Normal functions and the geometry of moduli spaces of curves.
\newblock In G.~Farkas and I.~Morrison, editors, {\em Handbook of {M}oduli},
  volume~I of {\em Advanced {L}ectures in {M}athematics}. International Press,
  2013.

\bibitem{MR735335}
J.~Harris.
\newblock On the {K}odaira dimension of the moduli space of curves. {II}. {T}he
  even-genus case.
\newblock {\em Invent. Math.}, 75(3):437--466, 1984.

\bibitem{MR664324}
J.~Harris and D.~Mumford.
\newblock On the {K}odaira dimension of the moduli space of curves.
\newblock {\em Invent. Math.}, 67(1):23--88, 1982.
\newblock With an appendix by William Fulton.

\bibitem{MR1908062}
E.-N. Ionel.
\newblock Topological recursive relations in {$H^{2g}(\M_{g,n})$}.
\newblock {\em Invent. Math.}, 148(3):627--658, 2002.

\bibitem{Marcus-Wise}
S.~Marcus and J.~Wise.
\newblock Stable maps to rational curves and the relative {J}acobian.
\newblock {\em Preprint, arXiv:1310.5981}, 2013.

\bibitem{RPnotes2013}
R.~Pandharipande.
\newblock Notes on classes related to equivalent divisors.
\newblock {\em {A}vailable at http://www.math.ethz.ch/~rahul/dzprop.pdf}, 2013.

\bibitem{MR2216265}
W.~F. Rulla.
\newblock Effective cones of quotients of moduli spaces of stable {$n$}-pointed
  curves of genus zero.
\newblock {\em Trans. Amer. Math. Soc.}, 358(7):3219--3237 (electronic), 2006.

\bibitem{MR3109733}
N.~Tarasca.
\newblock Brill--{N}oether loci in codimension two.
\newblock {\em Compos. Math.}, 149(9):1535--1568, 2013.

\end{thebibliography}

\end{document}